\documentclass[11pt]{article}
\usepackage{amsfonts}
\usepackage{latexsym}
\usepackage{cite}
\usepackage{amsmath,amsfonts,latexsym,amssymb,color}
\usepackage[mathscr]{eucal}
\usepackage{cases}
\usepackage{amsthm}

\usepackage[bf,small]{caption2}
\usepackage{float}
\usepackage{graphicx}
\usepackage{amsmath}
\usepackage{amssymb}
\usepackage[all]{xy}

\newtheorem{theorem}{theorem}[section]
\newtheorem{thm}[theorem]{Theorem}
\newtheorem{lem}[theorem]{Lemma}

\newtheorem{rmk}[theorem]{Remark}
\newtheorem{nota}[theorem]{Notation}

\begin{document}

\title{\textbf{The automorphism group of a nonsplit metacyclic 2-group}}
\author{\Large Haimiao Chen
\footnote{Email: \emph{chenhm@math.pku.edu.cn}}  \\
\normalsize \em{Mathematics, Beijing Technology and Business University, Beijing, China}}
\date{}
\maketitle

\begin{abstract}
  We determine the structure of automorphism group or each nonsplit metacyclic 2-group. This completes the work on automorphism groups of metacyclic $p$-groups.

\end{abstract}

\section{Introduction}

A {\it metacyclic group} is an extension of a cyclic group by another cyclic group.
As is well-known (see Section 3.7 of \cite{Zas56}), each metacyclic group can be presented as
\begin{align}
H(n,m;t,r)=\langle \alpha,\beta\mid \alpha^{n}=1, \ \beta^{m}=\alpha^{t}, \ \beta\alpha\beta^{-1}=\alpha^{r}\rangle. \label{eq:presentation}
\end{align}
for some positive integers $n,m,r,t$ with $r^{m}-1\equiv t(r-1)\equiv 0\pmod{n}$.

In recent years, people have been studying automorphism groups of metacyclic $p$-groups. The order and the structure of ${\rm Aut}(G)$ when $G$ is a split metacyclic $p$-group were found in \cite{BC06,Cur07}; when $G$ is a nonsplit metacyclic $p$-group with $p\neq 2$, ${\rm Aut}(G)$ was determined in \cite{Cur08}.

In this paper, we deal with nonsplit metacyclic 2-groups, based on \cite{CXZ17} in which we derived explicit formulas for automorphisms of a general metacyclic group.
According to Theorem 3.2 of \cite{Kin73}, each nonsplit metacyclic 2-group is isomorphic to one of the following:
\begin{enumerate}
  \item[\rm(I)] $H_I(a,b,c,d):=H(2^{a}, 2^{b}; 2^{c}, 2^{d}+1)$ for a unique quadruple
        $(a,b,c,d)$ with
        \begin{align}
        d>1  \qquad \text{and}  \qquad \max\{d,a-d-1\}<c<\min\{a,b\},   \label{ineq:ordinary}
        \end{align}
  \item[\rm(II)] $H_{II}(a,b,e):=H(2^{a}, 2^{b}; 2^{a-1}, 2^{e}-1)$ for a unique triple $(a,b,e)$ with
        \begin{align}
        \max\{1,a-b\}<e<\min\{a,b\},    \label{ineq:exceptional}
        \end{align}
  \item[\rm(III)] $H_{III}(a):=H(2^{a}, 2; 2^{a-1}, 2^{a}-1)$ for a unique $a>1$.
\end{enumerate}
We do not deal with case (III), in which $G$ is a {\it generalized quaternion group}, and the automorphism group was determined in \cite{ZZ05}. Denote $r=2^{d}+1$ in case (I), and $r=2^{e}-1, c=a-1, d=1$ in case (II). Denote $H(2^{a},2^{b};2^{c},r)$ simply as $H$ whenever there is no ambiguity.


\begin{nota}
\rm For an integer $N>0$, let $\mathbb{Z}_{N}$ denote $\mathbb{Z}/N\mathbb{Z}$, and  regard it as a quotient ring of $\mathbb{Z}$.
Let $\mathbb{Z}_N^{\times}$ denote the multiplicative group of units.
For $u\in\mathbb{Z}$, denote its image under the quotient $\mathbb{Z}\twoheadrightarrow\mathbb{Z}_{N}$ also by $u$.

For integers $u,s>0$, let $[u;s]=1+s+\cdots+s^{u-1}$ and $[0;s]=0$.

For any $k,v\in\mathbb{Z}$, by $r^k$ we mean $r^{\check{k}}$, where $\check{k}\in\{0,1,\ldots,2^{b}-1\}$ denotes the remainder when dividing $k$ by $2^b$; by $[v;r^k]$ we mean $[\check{v};r^{\check{k}}]$. For $w\in\mathbb{Z}_{2^a}$, by $\alpha^w$ we mean $\alpha^{\tilde{w}}$ where $\tilde{w}$ is any integer mapped to $w$ by $\mathbb{Z}\twoheadrightarrow\mathbb{Z}_{2^a}$; similarly, $\beta^w$ for $w\in\mathbb{Z}_{2^b}$ makes sense.

For an integer $u\ne 0$, let $\|u\|$ denote the largest integer $s$ with $2^{s}\mid u$; set $\|0\|=+\infty$.

Use $\exp_{\alpha}(u)$ to denote $\alpha^{u}$ when the expression for $u$ is too long.

Throughout the paper, we abbreviate $A\equiv B\pmod{2^a}$ to $A\equiv B$ as often as possible.
\end{nota}

\section{Preparation}

It follows from (\ref{eq:presentation}) that
\begin{align}
(\alpha^{u_{1}}\beta^{v_{1}})(\alpha^{u_{2}}\beta^{v_{2}})&=\alpha^{u_{1}+u_2r^{v_{1}}}\beta^{v_{1}+v_{2}},\label{eq:identity1} \\
(\alpha^{u}\beta^{v})^{k}&=\alpha^{u[k;r]}\beta^{vk}. \label{eq:identity2}
\end{align}

For the following two lemmas, see \cite{CXZ17} Lemma 2.1--2.4.
\begin{lem} \label{lem:simplify}
If $s>1$ with $\|s-1\|=\ell>1$ and $x>0$ with $\|x\|=u>0$, then
$$[x;s]\equiv (1+2^{\ell-1})x\pmod{2^{\ell+u}},  \qquad s^{x}-1\equiv (s-1+2^{2\ell-1})x\pmod{2^{2\ell+u}}.$$
\end{lem}

\begin{lem}  \label{lem:well-defined}
There exists $\sigma\in{\rm Aut}(H)$ with
$\sigma(\alpha)=\alpha^{x_{1}}\beta^{y_{1}}$, $\sigma(\beta)=\alpha^{x_{2}}\beta^{y_{2}}$ if and only if
\begin{align}
2^{b-d}\mid y_1  \label{ineq:deg-y1} \\
x_{2}[2^{b};r^{y_{2}}]+2^cy_{2}-x_{1}[2^{c};r^{y_{1}}]-\frac{2^{c}y_{1}}{2^b}2^c\equiv 0, \label{eq:condition2}\\
(r^{y_{1}}-1)x_2+([r;r^{y_{1}}]-r^{y_{2}})x_1+\frac{(r-1)y_1}{2^b}2^c\equiv 0,    \label{eq:condition3}  \\
2\nmid x_{1}y_{2}-x_{2}y_{1}.  \label{eq:condition0}
\end{align}
\end{lem}

Due to (\ref{ineq:deg-y1}), we may write $y_1=2^{b-d}y$.
The following two lemmas are special cases of \cite{CXZ17} Lemma 2.8 (iii) and (ii), respectively. Here we reprove them in a relatively succinct way. 
\begin{lem}
When $H=H_I(a,b,c,d)$, the conditions {\rm(\ref{eq:condition2})} and {\rm(\ref{eq:condition3})} hold if and only if
\begin{align}
x_{1}&\equiv 1+2^{b-c}x_{2}\pmod{2^{a-c}},  \label{eq:x12} \\
y_{2}&\equiv 1+(2^{c-d}+2^{c-1})y \pmod{2^{a-d}}.   \label{eq:y12}
\end{align}
\end{lem}

\begin{proof}
Now that $\|y_{1}\|\geq b-d>c-d\ge a-2d$, applying Lemma \ref{lem:simplify}, we obtain
$r^{y_{1}}\equiv 1+2^{b}y$, $[r]_{r^{y_{1}}}\equiv r$, $[2^{c}]_{r^{y_{1}}}\equiv 2^{c}$, $[2^{b}]_{r^{y_{2}}}\equiv 2^{b}$,
hence  (\ref{eq:condition2}) and (\ref{eq:condition3}) become
\begin{align}
2^{b}x_{2}+2^{c}(y_{2}-x_{1}-2^{c-d}y)&\equiv 0, \label{eq:csq1}   \\
2^{b}x_{2}y+(r-r^{y_{2}})x_{1}+2^{c}y&\equiv 0,  \label{eq:csq2}
\end{align}
respectively.
It follows from (\ref{eq:csq2}) that
\begin{align}
d+\|y_{2}-1\|=\|r^{y_2-1}-1\|\ge c\ge a-d;   \label{eq:d(y2-1)}
\end{align}
by Lemma \ref{lem:simplify}, $r^{y_2-1}-1\equiv 2^{d}(1+2^{d-1})(y_{2}-1)$, hence
$$r^{y_{2}}-r=(1+2^d)(r^{y_2-1}-1)\equiv (2^{d}+2^{2d-1})(y_{2}-1).$$
So (\ref{eq:csq2}) can be converted into
$(1+2^{d-1})(y_{2}-1)x_1\equiv (2^{b-d}x_{2}+2^{c-d})y\pmod{2^{a-d}};$
multiplying $1-2^{d-1}$ and using (\ref{eq:d(y2-1)}), we obtain
$$x_{1}(y_{2}-1)\equiv (2^{b-d}x_{2}+2^{c-d}-2^{c-1})y\pmod{2^{a-d}}.$$
Then
\begin{align*}
(x_{1}+2^{c-d}y)(y_{2}-1-2^{c-d}y)
&=x_{1}(y_{2}-1)+2^{c-d}(y_{2}-x_{1}-2^{c-d}y-1)y  \\
&\equiv (2^{b-d}x_{2}-2^{c-1}+2^{c-d}(y_{2}-x_{1}-2^{c-d}y))y  \pmod{2^{a-d}} \\
&\equiv -2^{c-1}y  \pmod{2^{a-d}},   
\end{align*}
where in the last line, (\ref{eq:csq1}) is used.
Thanks to $2\nmid x_{1}+2^{c-d}y$ and $c\ge a-d$, we can deduce (\ref{eq:y12}), then (\ref{eq:x12}) follows from (\ref{eq:y12}), (\ref{eq:csq1}) and the condition $c-1\ge d\ge a-c$.

Conversely, it can be verified that (\ref{eq:x12}), (\ref{eq:y12}) indeed imply (\ref{eq:csq1}), (\ref{eq:csq2}).
\end{proof}

\begin{lem}  \label{lem:deduce2}
When $H=H_{II}(a,b,e)$, the conditions {\rm(\ref{eq:condition2})} and {\rm(\ref{eq:condition3})} hold if and only if one of the following occurs:
\begin{itemize}
  \item $2\mid y$ and $y_{2}\equiv 1\pmod{2^{a-e}}$;
  \item $e\le a-2$, $2\nmid y$ and $\|y_{2}-1\|=a-e-1$.
\end{itemize}
\end{lem}

\begin{proof}
For any $z$ with $\|z\|=u\ge 1$, applying Lemma \ref{lem:simplify} to $s=r^{2}$, $x=z$, we obtain
\begin{align}
[2z;r]=(1+r)[z;r^{2}]\equiv 2^{e}(1+2^{e})z\equiv 2^{e}z&\pmod{2^{e+u+1}}, \nonumber \\
r^{2z}-1=(r^2)^z-1\equiv (r^2-1+2^{2e+1})z&\pmod{2^{2e+2+u}}; \label{eq:simplify2}
\end{align}
these are still true when $r$ is replaced by $r^{y_{2}}$, as $2\nmid y_{2}$ so that $\|r^{y_{2}}\pm 1\|=\|r\pm 1\|$.
In particular,
\begin{align}
[2^{b};r^{y_{2}}]\equiv 2^{e+b-1}\equiv 0, \qquad  r^{y_{1}}-1\equiv 0. \label{eq:deduce2}
\end{align}

Note that $2^{2c-b}y_{1}\equiv 2^{c}(y_{2}-x_{1})\equiv 0$, hence (\ref{eq:condition2}) holds for free.
The condition (\ref{eq:condition3}) becomes
$r^{y_{2}-1}-1\equiv 2^{a-1}y$, which is, by (\ref{eq:simplify2}), equivalent to
$$e+\|y_{2}-1\|\ge a \quad \text{if} \quad 2\mid y, \qquad e+\|y_{2}-1\|=a-1 \quad \text{if} \quad  2\nmid y;$$
by (\ref{eq:condition0}), the second possibility occurs only when $e\le a-2$.
\end{proof}

\section{The structure of automorphism group}

\subsection{$H=H_I(a,b,c,d)$}

Let $\Xi$ denote the set of $(x_1,x_2)\in\mathbb{Z}_{2^a}\times\mathbb{Z}_{2^a}$ satisfying (\ref{eq:x12}), and let $\Omega$ denote the set of $(y,y_2)\in\mathbb{Z}_{2^{a+d-c}}\times\mathbb{Z}_{2^{a+b-c}}$ satisfying (\ref{eq:y12}).
By Lemma \ref{lem:simplify} and (\ref{ineq:ordinary}), we have
$$r^{2^{b-d}y}=(1+2^d)^{2^{b-d}y}\equiv 1+2^by.$$
Thus each automorphism of $H_I(a,b,c,d)$ can be expressed as
\begin{align}
\sigma_{x_{1},x_{2};y,y_{2}}:\ \alpha^{u}\beta^{v}\mapsto \exp_{\alpha}(x_{1}[u;1+2^by]+(1+2^byu)x_{2}[v;r^{y_{2}}])\exp_{\beta}(2^{b-d}yu+y_{2}v)
\end{align}
for some quadruple $(x_{1},x_{2},y,y_{2})\in\Xi\times\Omega$, and $\sigma_{x_{1},x_{2};y,y_{2}}=\sigma_{x'_{1},x'_{2};y',y'_{2}}$
if and only if
\begin{align}
y'\equiv y\pmod{2^d}, \qquad  x'_1-x_1+2^{c-d}(y'-y)\equiv 0,  \label{eq:auto1=1} \\
y'_2\equiv y_2\pmod{2^b}, \qquad x'_2-x_2+\frac{y'_2-y_2}{2^b}2^c\equiv 0. \label{eq:auto1=2}
\end{align}

Let
\begin{align}
\phi_{x_1,x_2}=\sigma_{x_1,x_2;0,1}, \qquad  \psi_{y,y_2}=\sigma_{1,0;y,y_2}.
\end{align}
The following can be verified using (\ref{eq:identity1}), (\ref{eq:identity2}):
\begin{align}
\phi_{x'_1,x'_2}\circ\phi_{x_1,x_2}&=\phi_{x'_1x_1,x'_1x_2+x'_2}, \\
\psi_{y',y'_2}\circ\psi_{y,y_2}&=\psi_{y'+y'_2y,y'_2y_2}, \\
\phi_{x_1,x_2}\circ\psi_{y,y_2}&=\sigma_{x_1+x_2[2^{b-d}y;r],x_2[y_2;r];y,y_2}, \\
\psi_{y,y_2}\circ\phi_{x_1,x_2}&=\sigma_{[x_1;r'],[x_2;r'];yx_1,2^{b-d}yx_2+y_2}, \qquad \text{with} \quad r'=1+2^by.
\end{align}

Let
\begin{align}
X=\{\phi_{x_1,x_2}\colon (x_1,x_2)\in\Xi\}, \qquad Y=\{\psi_{y_1,y_2}\colon (y,y_2)\in\Omega\}.
\end{align}
Then $X,Y$ are subgroups of ${\rm Aut}(H)$,  
and there is a decomposition
\begin{align}
&\sigma_{x_1,x_2;y,y_2}=\phi_{\tilde{x}_1,\tilde{x}_2}\circ\psi_{y,y_2}, \\
\text{with} \qquad  &\tilde{x}_1=x_1-x_2[2^{b-d}y;r][y_2;r]^{-1}, \qquad \tilde{x}_2=x_2[y_2;r]^{-1}.  \label{eq:decomposition}
\end{align}
As special cases of (\ref{eq:auto1=1}) and (\ref{eq:auto1=2}), $\phi_{x_1,x_2}=\psi_{y,y_2}$ if and only if 
\begin{align}
\|y\|\ge d, \qquad \|y_2-1\|\ge b, \qquad x_1\equiv 1+2^{c-d}y, \qquad x_2\equiv\frac{y_2-1}{2^b}2^c.
\end{align}
It follows that
\begin{align}
X\cap Y=\langle\phi_{1+2^c,0},\phi_{1,2^c}\rangle=\langle\psi_{2^d,1},\psi_{0,1+2^b}\rangle\cong\mathbb{Z}_{2^{a-c}}\times\mathbb{Z}_{2^{a-c}}.
\end{align}

Let
\begin{align}
f=\min\{a,b\}, \qquad z=\begin{cases} -1, & b<a, \\ 0, &b\ge a,\end{cases} \qquad w=-(1+2^{d-1})^{-1}, \\
\phi_0=\phi_{1,2^{a-f}}, \qquad \phi_1=\phi_{1-2^{f-c},z}, \qquad \psi_0=\psi_{2^{a-c},1}, \qquad \psi_1=\psi_{w,1-2^{c-d}}.
\end{align}

Consider the homomorphism
$\rho:X\to\mathbb{Z}_{2^a}$ sending $\phi_{x_1,x_2}$ to $x_1$. It is easy to see that
\begin{align}
\ker(\rho)&=\{\phi_{1,x_2}\colon\|x_2\|\ge a-f\}=\langle\phi_{0}\rangle\cong\mathbb{Z}_{2^{f}}, \\
{\rm Im}(\rho)&=\{x\colon \|x-1\|\ge f-c\}\le\mathbb{Z}_{2^a}^\times.
\end{align}
If $f\ge c+2$, then
${\rm Im}(\rho)=\langle 1-2^{f-c}\rangle\cong\mathbb{Z}_{2^{a+c-f}}$, hence
\begin{align}
X=\langle\phi_0,\phi_1\mid \phi_0^{2^{f}}, \phi_1^{2^{a+c-f}},\phi_1\phi_0\phi_1^{-1}\phi_0^{2^{f-c}-1}\rangle&\cong\mathbb{Z}_{2^f}\rtimes\mathbb{Z}_{2^{a+c-f}};
\end{align}
if $f=c+1$, then ${\rm Im}(\rho)=\mathbb{Z}_{2^a}^\times=\langle-1,5\rangle\cong\mathbb{Z}_2\times\mathbb{Z}_{2^{a-2}}$ by Theorem 2' on Page 43 of \cite{IR90}. Hence
\begin{align}
X=\langle\phi_0,\phi_1,\phi_2\mid \phi_0^{2^f}, \phi_1^2, \phi_2^{2^{a-2}}, (\phi_1\phi_0)^2, \phi_2\phi_0\phi_2^{-1}\phi_0^{-5},[\phi_1,\phi_2]\rangle&\cong\mathbb{Z}_{2^f}\rtimes(\mathbb{Z}_2\times\mathbb{Z}_{2^{a-2}}), \\
\text{with} \qquad  \phi_2&=\phi_{5,-2z}.
\end{align}

Consider the homomorphism
$\mu:Y\to\mathbb{Z}_{2^{b+a-c}}$ sending $\psi_{y_1,y_2}$ to $y_2$. Clearly
\begin{align}
\ker(\mu)&=\{\psi_{y,1}\colon y\in\mathbb{Z}_{2^{a+d-c}}\colon \|y\|\ge a-c\}=\langle\psi_{0}\rangle\cong\mathbb{Z}_{2^{d}}, \\
{\rm Im}(\mu)&=\{y_2\colon \|y_2-1\|\ge c-d\}\le\mathbb{Z}_{2^{a+b-c}}^\times.
\end{align}
If $c\ge d+2$, then ${\rm Im}(\mu)=\langle 1-2^{c-d}\rangle\cong\mathbb{Z}_{2^{a+b+d-2c}}$, hence
\begin{align}
Y=\langle\psi_0,\psi_1\mid \psi_0^{2^d}, \psi_1^{2^{a+b-d-2c}}, \psi_1\psi_0\psi_1^{-1}\phi_0^{2^{c-d}-1}\rangle\cong\mathbb{Z}_{2^{d}}\rtimes\mathbb{Z}_{2^{a+b+d-2c}};
\end{align}
if $c=d+1$, then ${\rm Im}(\mu)=\mathbb{Z}_{2^{b+a-c}}^\times=\langle -1,5\rangle\cong\mathbb{Z}_2\times\mathbb{Z}_{2^{a+b-c-2}}$, hence
\begin{align}
Y=\langle\psi_{0},\psi_1,\psi_2\mid \psi_0^{2^d}, \psi_1^{2}, \psi_2^{2^{a+b-c-2}}, (\psi_1\psi_0)^2,\psi_2\psi_0\psi_2^{-1}\psi_0^{-5},[\psi_1,\psi_2]\rangle
&\cong\mathbb{Z}_{2^{d}}\rtimes(\mathbb{Z}_2\times\mathbb{Z}_{2^{a+b-c-2}}), \\
\text{with} \qquad \psi_2&=\psi_{-2w,5}.
\end{align}

The above can be summarized as
\begin{thm}
For $H=H_I(a,b,c,d)$, we have ${\rm Aut}(H)=XY$ such that 
\begin{align*}
X&\cong\begin{cases} \mathbb{Z}_{2^f}\rtimes\mathbb{Z}_{2^{a+c-f}}, &f\ge c+2, \\ \mathbb{Z}_{2^f}\rtimes(\mathbb{Z}_2\times\mathbb{Z}_{2^{a-2}}), &f=c+1, \end{cases} \\
Y&\cong\begin{cases} \mathbb{Z}_{2^d}\rtimes\mathbb{Z}_{2^{a+b+d-2c}}, &c\ge d+2, \\ \mathbb{Z}_{2^d}\rtimes(\mathbb{Z}_2\times\mathbb{Z}_{2^{a+b-c-2}}), &c=d+1, \end{cases} \\ 
X\cap Y&\cong\mathbb{Z}_{2^{a-c}}\times\mathbb{Z}_{2^{a-c}}.
\end{align*}
Consequently, the order of ${\rm Aut}(H)$ is $2^{b+c+2d}$. 
\end{thm}

\begin{rmk}
\rm In principle, we are able to obtain a presentation for ${\rm Aut}(H)$. The generators are those of $X,Y$, and the relator set can be divided into $R_X\cup R_Y\cup R_{X,Y}\cup R'$, where
\begin{itemize}
  \item $R_X$ (resp. $R_Y$) consists of relators in the presentations for $X$ (resp. $Y$), 
  \item $R_{X,Y}$ consists of two elements corresponding to the two generators of $X\cap Y$, i.e., $\phi_{1+2^c,0}=\psi_{2^d,1}$ and $\phi_{1,2^c}=\psi_{0,1+2^b}$,
  \item relators in $R'$ have the form $\vartheta_j\theta_i=\gamma\delta$, where $\vartheta_j$ (resp. $\theta_i$) is a generator for $Y$ (resp. $X$), and $\gamma$ (resp. $\delta$) is a product of generators for $X$ (resp. $Y$).
\end{itemize}
As an example of element in $R'$, 
$$\psi_0\circ\phi_0=\sigma_{1,2^{a-f};2^{a-c},h}=\phi_{h^{-1},2^{a-f}h^{-1}}\circ\psi_{2^{a-c},h}, \qquad \text{with} \qquad h=1+2^{2a+b-c-d-f};$$
one can further write $\phi_{h^{-1},2^{a-f}h^{-1}}$ (resp. $\psi_{2^{a-c},h}$) as a product of generators for $X$ (resp. $Y$).

However, the computations are so complicated that we choose not to write down explicitly.
\end{rmk}

\subsection{$H=H_{II}(a,b,e)$}

Recall that $r=2^e-1$, $c=a-1$ and $d=1$.

Let $\Xi$ denote the set of $(x_1,x_2)\in\mathbb{Z}_{2^a}\times\mathbb{Z}_{2^a}$ with $2\nmid x_1$, and let $\Omega$ denote the set of $(y,y_2)\in\mathbb{Z}_{4}\times\mathbb{Z}_{2^{b+1}}$ satisfying the conditions in Lemma \ref{lem:deduce2}.
Each automorphism of $H_{II}(a,b,e)$ can be expressed as (recalling $r^{y_1}\equiv 1$, as in (\ref{eq:deduce2}))
\begin{align}
\sigma_{x_{1},x_{2};y,y_{2}}:\ \alpha^{u}\beta^{v}\mapsto \exp_{\alpha}(x_{1}u+x_{2}[v;r^{y_{2}}])\exp_{\beta}(2^{b-1}yu+y_{2}v)
\end{align}
for some quadruple $(x_{1},x_{2},y,y_{2})\in\Xi\times\Omega$, and $\sigma_{x_{1},x_{2};y,y_{2}}=\sigma_{x'_{1},x'_{2};y',y'_{2}}$ if and only if
\begin{align}
y'\equiv y\pmod{2}, \qquad x'_1-x_1+2^{a-2}(y'-y)\equiv 0, \label{eq:auto2=1} \\
y'_2\equiv y_2\pmod{2^b}, \qquad x'_2-x_2+\frac{y'_2-y_2}{2^b}2^{a-1}\equiv 0. \label{eq:auto2=2}
\end{align}

Let
\begin{align}
\phi_{x_1,x_2}=\sigma_{x_1,x_2;0,1}, \qquad \psi_{y,y_2}=\sigma_{1,0;y,y_2}.
\end{align}
We have
\begin{align}
\phi_{x'_1,x'_2}\circ\phi_{x_1,x_2}&=\phi_{x'_1x_1,x'_1x_2+x'_2}, \\
\psi_{y',y'_2}\circ\psi_{y,y_2}&=\psi_{y'+y'_2y,y'_2y_2}, \\
\phi_{x_1,x_2}\circ\psi_{y,y_2}&=\sigma_{x_1+x_2[2^{b-1}y;r],x_2[y_2;r];y,y_2}, \\
\psi_{y,y_2}\circ\phi_{x_1,x_2}&=\sigma_{x_1,x_2;yx_1,2^{b-1}yx_2+y_2}.
\end{align}
Let
\begin{align}
X=\{\phi_{x_1,x_2}\colon(x_1,x_2)\in\Xi\}, \qquad Y=\{\psi_{y,y_2}\colon(y,y_2)\in\Omega\}.
\end{align}
Thus $X,Y$ are subgroups of ${\rm Aut}(H)$, and there is a decomposition
\begin{align}
&\sigma_{x_1,x_2;y,y_2}=\phi_{\tilde{x}_1,\tilde{x}_2}\circ\psi_{y,y_2}, \\
\text{with} \qquad  &\tilde{x}_1=x_1-x_2[2^{b-1}y;r][y_2;r]^{-1}, \qquad \tilde{x}_2=x_2[y_2;r]^{-1}.  \label{eq:decomposition}
\end{align}
As special cases of (\ref{eq:auto2=1}), (\ref{eq:auto2=2}), $\phi_{x_1,x_2}=\psi_{y,y_2}$ if and only if 
\begin{align}
y\in\{0,2\}, \qquad y_2\in\{1,1+2^b\}, \qquad x_1\equiv 1+2^{a-2}y, \qquad x_2\equiv\frac{y_2-1}{2^b}2^{a-1}.
\end{align}
Hence
\begin{align}
X\cap Y=\{\psi_{0,1},\psi_{0,1+2^b},\psi_{2,1},\psi_{2,1+2^b}\}\cong\mathbb{Z}_2\times\mathbb{Z}_2.
\end{align}

Let
\begin{align}
\phi_0=\phi_{1,1}, \qquad \phi_1=\phi_{-1,0}, \qquad \phi_2=\phi_{5,0}, \qquad \psi_0=\psi_{2,1}, \qquad \psi_1=\psi_{1,1-2^{a-e-1}}.
\end{align}

Similarly as in the previous subsection, we can obtain
\begin{align}
X=\langle\phi_0,\phi_1,\phi_2\mid \phi_0^{2^a}, \phi_1^2, \phi_2^{2^{a-2}}, \phi_1\phi_0\phi_1^{-1}\phi_0,\phi_2\phi_0\phi_2^{-1}\phi_0^{-5},[\phi_1,\phi_2]\rangle
\cong\mathbb{Z}_{2^{a}}\rtimes(\mathbb{Z}_2\times\mathbb{Z}_{2^{a-2}}).
\end{align}

If $e\le a-3$, then
\begin{align}
Y=\langle\psi_0,\psi_1\mid \psi_0^2, \psi_1^{2^{b+e-a+2}}, [\psi_0,\psi_1]\rangle\cong\mathbb{Z}_{2}\times\mathbb{Z}_{2^{b+e-a+2}}.
\end{align}

If $e=a-2$, then
\begin{align}
Y=\langle\psi_1,\psi_2\mid \psi_1^2, \psi_2^{2^{b-1}},[[\psi_1,\psi_2],\psi_1],[[\psi_1,\psi_2],\psi_2],[\psi_1,\psi_2]^2\rangle, \qquad 
\text{with} \qquad \psi_2=\psi_{1,5}.
\end{align}
In the notations of \cite{Br82} Chapter IV, Section 3, choose a map 
$s:\mathbb{Z}_2\times\mathbb{Z}_{2^{b-1}}\to Y$, $(u,v)\mapsto \psi_1^u\psi_2^v$, so that $s(0,0)=1$, then it can be computed that
$$s(u_1,v_1)s(u_2,v_2)s(u_1+u_2,v_1+v_2)^{-1}=[\psi_1,\psi_2]^{u_2v_1}.$$
Hence $Y$ is isomorphic to the central extension of $\mathbb{Z}_2\times\mathbb{Z}_{2^{b-1}}$ by $\mathbb{Z}_2$ determined by the 2-cocycle
\begin{align}
f:(\mathbb{Z}_2\times\mathbb{Z}_{2^{b-1}})^2\to\mathbb{Z}_2, \qquad (u_1,v_1,u_2,v_2)\to u_2v_1;
\end{align}
denote this group by $E_f$.

If $e=a-1$, then we find directly that
\begin{align}
Y=\langle \psi_0,\tilde{\psi}_1,\tilde{\phi}_2\mid \psi_0^2, \tilde{\psi}_1^2,\tilde{\psi}_2^{2^{b-1}}, [\psi_0,\tilde{\psi}_1], [\psi_0,\tilde{\psi}_2], [\tilde{\psi}_1,\tilde{\psi}_2]\rangle&\cong \mathbb{Z}_2\times\mathbb{Z}_2\times\mathbb{Z}_{2^{b-1}}, \\
\text{with} \qquad \tilde{\psi}_1=\psi_{0,-1}, \qquad \tilde{\psi}_2&=\psi_{0,5}.
\end{align}

\begin{thm}
For $H=H_{II}(a,b,e)$, we have ${\rm Aut}(H)=XY$ such that
\begin{align*}
X&\cong \mathbb{Z}_{2^a}\rtimes(\mathbb{Z}_2\times\mathbb{Z}_{2^{a-2}}), \\
Y&\cong\begin{cases} \mathbb{Z}_{2}\times\mathbb{Z}_{2^{b+e-a+2}}, &e\le a-3, \\ 
E_f, &e=a-2, \\
\mathbb{Z}_2\times\mathbb{Z}_2\times\mathbb{Z}_{2^{b-1}}, &e=a-1, \end{cases} \\
X\cap Y&\cong\mathbb{Z}_{2}\times\mathbb{Z}_{2}.
\end{align*}
Consequently, the order of ${\rm Aut}(H)$ is $2^{a+b+\min\{a-2,e\}}$.
\end{thm}


\begin{thebibliography}{}


\bibitem{BC06}
J.N.S. Bidwell and M.J. Curran,
\textsl{The automorphism group of a split metacyclic $p$-group}.
Arch. Math. 87 (2006), 488--497.


\bibitem{Br82}
K.S. Brown,
\textsl{Cohomology of groups}.
Graduate Texts in Mathematics Vol 87, Springer-Verlag, New York, Heidelberg, Berlin, 1982.



\bibitem{CXZ17}
H.-M. Chen, Y.-S. Xiong, Z.-J. Zhu,
\textsl{Automorphism of metacyclic groups}.
arXiv:1506.02234, accepted by Czechoslovak J. Math.


\bibitem{Cur07}
M.J. Curran,
\textsl{The automorphism group of a split metacyclic 2-group}.
Arch. Math. 89 (2007), 10--23.

\bibitem{Cur08}
M.J. Curran,
\textsl{The automorphism group of a nonsplit metacyclic p-group}.
Arch. Math. 90 (2008), 483--489.






\bibitem{IR90}
K. Ireland, M. Rosen,
A Classical Introduction to Modern Number Theory,
Graduate Texts in Mathematics Vol 84, 2nd edition, Springer-Verlag, New York, 1990.


\bibitem{Kin73}
B.W. King,
\textsl{Presentations of metacyclic groups}.
Bull. Austral. Math. Soc. 8 (1973), 103--131.

\bibitem{Zas56}
H.J. Zassenhaus,
\textsl{The theory of groups}.
Second edition, Chelsea, New York, 1956.

\bibitem{ZZ05}
D.-G. Zhu and G.-X. Zuo,
\textsl{The Automorphism Group and Holomorph of Quaternion Group(in generalized sense)}.
Acta Mathematiea Scientia, 25 (2005), no. 1, 79--83.

\end{thebibliography}
\end{document}